\documentclass[12pt]{article}

\usepackage{epsfig,amsmath,amsfonts,amssymb,latexsym,color,amsthm,hhline,multirow,subfigure,graphicx}
\usepackage{amsmath}
\usepackage{amsthm}
\usepackage{enumerate}
\usepackage{amsfonts}
\usepackage{verbatim}
\usepackage{graphicx}
\usepackage{float}

\newtheorem{theorem}{Theorem}[section]
\newtheorem{prop}{Proposition}[section]
\newtheorem{lemma}{Lemma}[section]
\newtheorem{remark}{Remark}[section]

\newcommand{\E}{{\mathbb E}}
\newcommand{\RR}{{\mathbb R}}
\newcommand{\Z}{{\mathbb Z}}
\newcommand {\PP}{{\mathbb P}}
\newcommand{\sss}{\scriptscriptstyle}
\newcommand{\Zd}{\mathbb{Z}^d}
\newcommand{\N}{\mathbb{N}}
\newcommand{\0}{{\bf 0}}
\newcommand{\1}{{\bf 1}}
\newcommand{\cA}{\mathcal{A}}
\newcommand{\cD}{\mathcal{D}}
\newcommand{\vep}{\varepsilon}
\newcommand{\n}{{\bf n}}
\newcommand{\m}{{\bf m}}

\begin{document}

\title{Competing frogs on $\Zd$}\parskip=5pt plus1pt minus1pt \parindent=0pt
\author{Maria Deijfen\thanks{Department of Mathematics, Stockholm University; {\tt mia@math.su.se}} \and Timo Hirscher \thanks{Department of Mathematics, Stockholm University; {\tt timo@math.su.se}} \and Fabio Lopes \thanks{Instituto de Estad\'{i}stica, Pontificia Universidad Cat\'{o}lica de Valpara\'{i}so; {\tt fabio.lima@pucv.cl}} }
\date{January 2019}
\maketitle

\begin{abstract}
\noindent A two-type version of the frog model on $\Zd$ is formulated, where active type $i$ particles move according to lazy random walks with probability $p_i$ of jumping in each time step ($i=1,2$). Each site is independently assigned a random number of particles. At time 0, the particles at the origin are activated and assigned type 1 and the particles at one other site are activated and assigned type 2, while all other particles are sleeping. When an active type $i$ particle moves to a new site, any sleeping particles there are activated and assigned type $i$, with an arbitrary tie-breaker deciding the type if the site is hit by particles of both types in the same time step. We show that the event $G_i$ that type $i$ activates infinitely many particles has positive probability for all $p_1,p_2\in(0,1]$ ($i=1,2$). Furthermore, if $p_1=p_2$, then the types can coexist in the sense that $\PP(G_1\cap G_2)>0$. We also formulate several open problems. For instance, we conjecture that, when the initial number of particles per site has a heavy tail, the types can coexist also when $p_1\neq p_2$.
\vspace{0.3cm}


\vspace{0.2cm}

\end{abstract}

\section{Introduction}\label{sec:intro}

The so called frog model on $\Zd$ is driven by moving particles on the sites of the $\Zd$-lattice. Each site $x\in\Zd$ is assigned an initial number $\eta(x)$ of particles, where $\{\eta(x)\}_{x\in\Zd}$ are independent and identically distributed. We write $\nu$ for the product measure defined by this initial particle distribution. Each particle is then independently equipped with a discrete time simple symmetric random walk, denoted for particle $j=1,\ldots,\eta(x)$ at the site $x$ by $(S^{x,j}_n)_{n\in\N}$ and encoded by jumps rather than sites. A particle starts moving from its initial location and the associated random walk then specifies the movement of the particle in each time step. The set of all these random walks is denoted by $S=\{(S^{x,j}_n)_{n\in\N}:x\in\Zd,j=1,\ldots,\eta(x)\}$. At time 0, the particles at the origin are activated, while all other particles are sleeping. When a particle is activated, it starts moving according to its associated random walk so that, in each time step, it moves to a uniformly chosen neighboring site. When a site is visited by an active particle, any sleeping particles at the site are activated and start moving. If the origin is non-empty, this means that the set of activated particles grows to infinity.

The model has previously been studied e.g.\ with respect to transience/recurrence \cite{telcs}, the shape of the set of visited sites \cite{frogs_shape, frogs_shape_random} and extinction/survival for a version of the model including death of active particles \cite{phase_transition}. Here we introduce a two-type version of the model, where an active particle can be of either of two types. We study the possibility for the types to activate infinitely many particles and investigate in particular the event of coexistence, which is said to occur if both types activate infinitely many particles. Similar questions have been studied for other competition models on $\Zd$, for instance the so-called Richardson model where a site becomes type $i$ infected ($i=1,2$) at a rate proportional to the number of nearest neighbor sites of type $i$. In our model however, the type is associated with the moving particles rather than the sites.

\subsection{Definition of the model}\label{sec:definition}

To define the model, first assign an initial number of particles per site according to the product measure $\nu$ and equip each particle with a random walk from the set $S$, as described above. At time 0, the particles at the origin are activated and assigned type 1, while the particles at another site $z\in\Zd$ are activated and assigned type 2. All other particles are sleeping and do not yet have a type. The activated particles then move according to their associated random walks in $S$. A type $i$ particle makes a jump in a given time step independently with probability $p_i$ and stays at its present location with probability $1-p_i$. When a particle leaves its location after a geometrically distributed number of time steps, it jumps to the next location in its associated random walk.

We say that a site is \emph{discovered} when it is first hit by an active particle. It is said to be \emph{discovered by type $i$} if the first particle(s) that hits it is of type $i$. Note that a site can be discovered by both types -- this happens if particles of both types arrive at the site in the same time step. If there are sleeping particles at the discovered site, these are activated and assigned the same type as the active particle(s) that discovered the site. If the site is discovered by both types, we fix an arbitrary rule for deciding the type(s) of its particles. We may e.g.\ toss a coin (fair or biased), assign the type(s) based on the number of particles of each type that discover the site, or deterministically always decide in favor of a given type. All our results hold for any tie-breaker rule; see however Section \ref{sec:open} for a discussion on potential effects. Once it has been activated and assigned a type, a particle remains active and keeps its type forever.

Formally, we construct the process as follows. Let $(x,j)$ denote particle $j$ at the site $x\in\Zd$ and let $(L^{x,j}_{n,k})_{n,k\in\N}$ be a family of independent and identically distributed (i.i.d.) random variables associated with the particle $(x,j)$, where $L^{x,j}_{n,k}$ is uniform on $[0,1]$. Write $L=\{(L^{x,j}_{n,k})_{n,k\in\N}:x\in\Zd,j=1,\ldots,\eta(x)\}$. These variables control the delays of the particles compared to their associated random walks: Assume that a particle $(x,j)$ has made $n$ jumps since it was activated, and that the particle arrived at its current location $S^{x,j}_n$ at time $t$. Its next move (to $S^{x,j}_{n+1}$) occurs at time $t+k$ if and only if $L^{x,j}_{n,m}>p$ for all $m<k$ and $L^{x,j}_{n,k}\leq p$. The randomness in the process is hence summarized by $\Pi=(\nu,S,L)$. The rule for breaking ties may incorporate additional randomness, which we omit in the notation since it will not play a role for our arguments. Write $\PP_{\0,z}$ for the probability measure of the process started at time 0 with the particles at the origin $\0$ type 1 and the particles at $z$ type 2.

Before proceeding, we note that $(\nu,S,L)$ can be used to formally construct a one-type process based on lazy random walks with probability $p$ of jumping in each time step. Both for the one-type process and the two-type process, the construction provides a coupling of the processes for different values of $p$ and $(p_1,p_2)$, respectively, where the set of discovered sites increases with $p$ in the one-type process and the set of sites discovered by type 1 (2) increases with $p_1$ ($p_2$) in the two-type process if $p_2$ ($p_1$) is kept fixed. By symmetry, we may assume that $p_2\leq p_1$ in the two-type process.

\subsection{Results}\label{sec:results}

It follows from the results in \cite{telcs} that the time until any given site is discovered is finite almost surely. All particles will hence eventually be activated. Our first result is that both types have a strictly positive probability of outcompeting the other type, in the sense that it activates infinitely many particles, while the other type activates only finitely many. If the initial particle distribution allows for empty sites, this is trivial -- since the starting site of either type may then be empty thereby preventing the type from growing at all -- but we show that it is true also conditioning on a non-zero number of particles on both starting sites. Intuitively, the winning type then manages to capture all particles in an area that surrounds all particles of the other type and that is thick enough to prevent the surrounded type from traversing it. The event that infinitely many particles are activated by type $i$ is denoted by $G_i$ and $G_i^c$ denotes its complement.

\begin{prop}\label{pr:both_can_win}
For any initial distribution $\nu$, any $p_1,p_2\in(0,1]$ and any $z\in\Zd$, conditional on $\eta(\0)\geq 1$ and $\eta(z)\geq 1$, we have that $\PP_{\0,z}(G_1\cap G^c_2)>0$ and $\PP_{\0,z}(G^c_1\cap G_2)>0$.
\end{prop}

Next we turn to the event $G_1\cap G_2$ that both types activate infinitely many sites. This corresponds to a power balance between the types in the sense that none of them manages to outcompete the other. We first show that whether this event has positive probability or not does not depend on the choice of the starting site $z$ for type 2 when $p_1\in(0,1)$. We may hence assume that type 2 starts at the site $\1=(1,0,\dots,0)\in\Zd$ next to the origin. We expect this to be true also when $p_1=1$, but the proof is based on a coupling argument that requires that particles can stay put in a given time step. It turns out however that a slight modification of the argument gives a weaker version when $p_1=1$; see Lemma \ref{le:odd_z}. This will be used to establish our main result when $p_1=p_2=1$.

\begin{prop}\label{pr:initial}
For any initial distribution $\nu$ and any $p_1,p_2\in(0,1)$, we have for any $z\in\Zd$ that $\PP_{\0,z}(G_1\cap G_2)>0$ if and only if $\PP_{\0,\1}(G_1\cap G_2)>0$.
\end{prop}

Our main result is that coexistence has a strictly positive probability when $p_1=p_2$. We are convinced that this is true for any initial distribution, but the possibility of having empty sites causes some technical problems that we are only able to handle when the expected initial number of particles per site is finite.

\begin{theorem}\label{th:coex}
Assume that either $\eta(x)\geq 1$ almost surely or $\E[\eta(x)]<\infty$. Then $\PP_{\0,\1}(G_1\cap G_2)>0$ if $p_1=p_2\in(0,1]$.
\end{theorem}

In other competition models on $\Zd$, the typical picture is that two species can coexist if and only if they are identical in the sense that they grow according to the same dynamics with the same parameter values. One might guess that the situation is similar here so that coexistence is not possible when $p_1\neq p_2$ and the types can hence coexist if and only if $p_1=p_2$. However, we do not think this is true. In particular, we think that, when the initial distribution has a very heavy tail, then the types can coexist for all values of $p_1$ and $p_2$. We comment further on this in Section \ref{sec:open}.

An important ingredient in the proof of all our results is the shape theorem for the one-type frog model. This was established in \cite{frogs_shape} starting with one particle per site and generalized in \cite{frogs_shape_random} to arbitrary initial distributions. Both versions concern the one-type model based on non-lazy random walks, but we will need the result also for a lazy version of the process, where the particles move according to lazy random walks with a probability $p\in(0,1]$ of jumping in each time step. This follows from the same proof as in \cite{frogs_shape,frogs_shape_random}; see the appendix.

To formulate the theorem, let $\xi_n(p)$ denote the set of discovered sites in a one-type process started from the origin where all particles move according to lazy random walks that have probability $p$ of moving in each time step. Formally, we use the family $L$ introduced above to control the delays of the random walks in $S$ to obtain the movements of the particles. Write $\bar{\xi}_n(p)=\{x+(1/2,1/2]^d:x\in\xi_n(p)\}$.

\begin{theorem}[General shape theorem]\label{th:shape}
For any $\nu$ and $p\in(0,1]$, there exists a non-empty convex set $\cA=\cA(\nu,p)$ such that, conditional on $\eta(\0)\geq 1$ and for any $\vep\in(0,1)$, almost surely
$$
(1-\vep)\cA\subset\frac{\bar{\xi}_n(p)}{n}\subset(1+\vep)\cA
$$
for large $n$.
\end{theorem}

Characterizing the shape $\cA$ largely remains an open problem. However a few things can be said about how the shape depends on the initial distribution $\nu$ and the parameter $p$.  By construction of the process, we have that $\xi_n(p)\subseteq\xi_n(p')$ for $p\leq p'$ and thereby $\cA(\nu,p)\subseteq\cA(\nu,p')$ for any $\nu$. For $x\in\RR^d$, let $\|x\|_1$ denote the $L_1$-norm of $x$ and let $\cD=\{x\in\RR^d:\|x\|_1\leq 1\}$. Due to the discrete nature of the model, the shape cannot exceed $\cD$, that is, $\cA(\nu,p)\subseteq\cD$ for any $\nu$ and $p$. In \cite{frogs_shape_random}, it is shown that, if $\nu$ is such that the initial number of particles $\eta(x)$ per site $x$ has a heavy tail, then $\cA(\nu,1)=\cD$. A minor modification of that proof shows that the conclusion remains valid also for $p<1$; see the appendix for a brief outline. Intuitively, if there are very many particles per site then, with overwhelming probability, one particle will jump to each neighbor in a given step even if the probability of jumping per particle is small.

\begin{theorem}\label{th:diamond}
Assume that $\nu$ satisfies $\PP(\eta(x)\geq n)\geq (\log n)^{-\delta}$ for some positive $\delta<d$ and all $n$ large enough. Then $\cA(\nu,p)=\cD$ for any fixed $p\in(0,1]$.
\end{theorem}

We describe possible implications of this result for the possibility of coexistence in Section \ref{sec:open} below, where we have collected open problems and suggestions for further work. Section \ref{sec:related} contains references to previous work on competition on $\Zd$. Proposition \ref{pr:both_can_win} and \ref{pr:initial} are then proved in Section \ref{sec:prop_both} and Section \ref{sec:prop_initial}, respectively, and Theorem \ref{th:coex} is proved in Section \ref{sec:th_proof}. Finally, some details on how Theorem \ref{th:shape} and Theorem \ref{th:diamond} are derived using their counterparts for $p=1$ are given in the appendix.

\subsection{Open problems}\label{sec:open}

Here we describe some open problems for the model, and some modifications of the model that might be worth further study.

\noindent \textbf{Coexistence and the shape.} A natural question is if Theorem \ref{th:coex} has a counterpart for $p_2<p_1$ saying that coexistence is then impossible. We do not think that this is the case. Instead, we expect that two types can coexist if and only if their one-type shapes coincide, that is, if $\nu$ and $(p_1,p_2)$ are such that $\cA(\nu,p_1)=\cA(\nu,p_2)$. According to Theorem \ref{th:diamond}, for any $p_1,p_2\in(0,1]$, a process with $p=p_1$ and one with $p=p_2$ both give the same maximal shape $\cD$ when $\nu$ has a sufficiently heavy tail, indicating that type 1 can then coexist with a strictly weaker type 2 if our intuition is correct.

To establish our intuition, one would have to show that, if the type 1 shape is strictly larger than the type 2 shape and type 1 activates infinitely many particles, then type 1 will sooner or later use its larger speed to surround type 2. To do this, one might try to generalize arguments used for first passage percolation; see e.g.\ \cite{cont_comp,HP2}. They are however incomplete in the sense that they cannot rule out the possibility of type 1 surviving in a weak sense, that is, growing unboundedly but occupying only a vanishing fraction of the active sites.

For a given initial distribution $\nu$, how is the shape affected by $p$? This is of independent interest, but would also be worth studying in view of its potential relevance for the possibility of coexistence. As pointed out above, we have $\cA(\nu,p)\subseteq\cA(\nu,p')\subseteq \cD$ for $p\leq p'$. Are there conditions on $\nu$ that guarantee that $\cA(\nu,p)$ is strictly smaller than $\cA(\nu,p')$ for all $p<p'$? According to Theorem \ref{th:diamond}, the $p$-shape and the $p'$-shape coincide when $\nu$ is heavy-tailed since the asymptotic growth rate in both cases is maximal. Are there cases when $\cA(\nu,p')$ is strictly smaller than $\cD$, and we still have $\cA(\nu,p)=\cA(\nu,p')$ for $p<p'$ sufficiently close to $p'$?

\noindent \textbf{The tie-breaker.} All our results apply for any tie-breaking rule. An unfair tie-breaker can hence not ruin the possibility of coexistence when $p_1=p_2$. One can still ask if an unfair tie-breaker can make coexistence possible in a situation where it is not possible with a fair tie-breaker (by giving an advantage to the weaker type). We think that the answer is no. However the tie-breaker could potentially influence the geometry of the sets of sites discovered by the respective types and the properties of the boundaries between them.

\noindent \textbf{Collective laziness.} An alternative way of modeling the delays is to toss one single coin in each time step deciding if the type 1 particles move or not, that is, with probability $p_1$ all type 1 particles move and with probability $1-p_1$ they all stay where they are. Similarly, one single coin toss determines if the type 2 particles move or not. The intuition behind our conjecture that coexistence is possible if and only if the one-type shapes coincide is fairly general and apply also to this version of the model. One might hence guess that it is qualitatively similar to our model, where the delays of the particles are independent. Note however that, with collective laziness the one-type shape theorem follows immediately from the version without laziness from a simple time-scaling argument and we obtain in this case that $\cA(\nu,p_2)=\frac{p_2}{p_1}\cA(\nu,p_1)$ for all $\nu$, so that type 2 gives rise to a strictly smaller shape when $p_2<p_1$. If indeed coexistence is possible if and only if the one-type shapes coincide for both versions of the model, then there will be choices for $(p_1,p_2)$ for which the types can coexist with independent laziness but not with collective laziness.

\noindent \textbf{Continuous time.} The frog model is traditionally studied in discrete time, but it could of course also be defined in continuous time by letting the particles move according to independent simple random walks in continuous time. This has been done in \cite{combustion}, where a shape theorem is proved for the one particle per site initial configuration $\eta(x)\equiv 1$. A two-type version of such a model would be obtained by letting type 1 particles jump with rate 1 and type 2 particles with rate $\beta<1$. It would have the advantage that no tie-breaker is needed, since particles will almost surely not jump simultaneously. For $\eta(x)\equiv 1$ we conjecture that coexistence is possible if and only if $\beta=1$. For other initial distributions, one would first have to establish a shape theorem. In contrast to the discrete case, this might require conditions on the initial distribution. Could the growth be superlinear in time if $\eta(x)$ has a very heavy tail? For $\nu$ that do give rise to a bounded shape, it follows by time-scaling that the shape at rate $\beta<1$ is strictly smaller than the rate 1 shape and we therefore conjecture that coexistence is possible if and only if $\beta=1$.

\subsection{Related work}\label{sec:related}

Competition models on $\Zd$ have been an active research area the last decades. A two-type version of the Richardson model was introduced in \cite{HP1}, with two types competing to invade the sites of the $\Zd$-lattice. The growth is driven by exponential passage times on the edges with potentially different intensities for the types, and the conjecture is that the types can grow to occupy infinitely many sites simultaneously if and only if they spread with the same intensity. The if-direction was proved in \cite{HP1} for $d=2$ and independently in \cite{GM_coex} and \cite{Hoff_coex} for $d\geq 2$. The only-if direction is not proved, but partial results can be found in \cite{GM_invisible,HP2}.

A variation of the two-type Richardson model, where a site that has at least two neighbors of a given type is immediately occupied by that type, is studied in \cite{urns}. Another variation was recently introduced in \cite{VA}. A type 1 process there starts from the origin and each time it reaches a new site, with some probability instead a type 2 process starts at this site. We also mention the multi-type
contact process, introduced in \cite{Neu} and further studied e.g.\ in \cite{MV, MPV}. There sites can recover and become susceptible again, and the focus is on properties of stationary measures.

In the above models, the type is associated with the sites. The frog model however is driven by moving particles, and the type in our two-type version is associated with the particles. A model that is related to the frog model is obtained by letting all particles move, that is, there are no sleeping particles but all particles start moving according to independent random walks at time 0. This model is technically considerably more challenging to analyze and has been studied e.g.\ in \cite{KS_05,KS_08}. A competition version was studied in \cite{RW_competition}. There the two types both move at rate 1, type 1 starting from a single site and type 2 from some infinite set of sites $S$, and a particle changes type if a particle of the other type jumps onto it. The main result is a condition on the set $S$ that determines when type 1 has a chance of surviving. In this context we also mention \cite{KS_03}. The model studied there is not a competition model, but also deals with the evolution of two interacting types.

\section{Proof of Proposition \ref{pr:both_can_win}}\label{sec:prop_both}

In this section we prove Proposition \ref{pr:both_can_win}. A key observation is that a given particle will almost surely discover only finitely many sites, implying that it will activate finitely many other particles. This follows from the fact that the  distance of a random walk from its starting point after $n$ steps scales like $\sqrt{n}$, while the set of discovered sites in the frog model grows linearly in $n$ according to the shape theorem. The proposition follows from this in combination with coupling arguments.

Consider a (possibly lazy) simple symmetric random walk $S_n$ on $\Zd$, starting at the origin. It is well known that the distance to the origin scales like $\sqrt{n}$. Let $\cD_r=\{x\in\RR^d:\;\lVert x\rVert_1\leq r\}$. The following result quantifies the probabilities of moderate deviations for the walk.

\begin{lemma}\label{le:SRW}
For any $\varepsilon\in(0,1/2)$, there exists a constant $\gamma>0$ such that, for all $n$:
\begin{equation}\label{eq:SRW}
\PP(S_n\in\cD_{n^{1-\varepsilon}})\geq 1-\exp\{-\gamma n^{1-2\varepsilon}\}.
\end{equation}
\end{lemma}

\begin{proof}
For a one-dimensional walk, it is proved in \cite{moderate} that $\PP(S_n\not\in \cD_{cn^{1-\varepsilon}})\leq \exp\{-\gamma n^{1-2\varepsilon}\}$
for $\varepsilon\in(0,1/2)$, all $c>0$ and some $\gamma>0$. This immediately gives the bound for $d=1$. For $d\geq 2$, the probability of a given jump being along the $x_k$-direction ($k=1,\ldots,d$) is $1/d$. The displacement in a given direction can hence be controlled by the one-dimensional bound, and by a union bound we obtain \eqref{eq:SRW} for the $d$-dimensional walk with $\cD_{n^{1-\varepsilon}}$ replaced by a cube with side length $2cn^{1-\varepsilon}$ centered at the origin. The desired bound follows from this by choosing $c$ small such that this cube is contained in $\cD_{n^{1-\varepsilon}}$.
\end{proof}

We now combine this with the shape theorem to conclude that any given particle discovers only finitely many sites.

\begin{lemma}\label{le:finite}
For any initial distribution, the number of sites discovered by a given particle in the one-type or two-type frog model is almost surely finite.
\end{lemma}

\begin{proof}
We show the statement for the origin particles in a one-type model, and then explain how this gives the general statement. Consider one of the initially activated particles at the origin in a one-type model. By Lemma \ref{le:SRW} and the Borel-Cantelli lemma, the position of the particle will almost surely be contained in $\cD_{n^{3/4}}$ for large $n$. However, by Theorem \ref{th:shape}, the one-type process grows linearly in $n$, and gives rise to a deterministic shape $\cA$ on the scale $n^{-1}$. The shape $\cA$ is non-empty and convex, implying that $\cA\supset \cD_\delta$ for some small $\delta>0$. Hence almost surely $\cD_{n\delta/2}\subset \xi_n$ for large $n$. It follows that the origin particle will almost surely not discover any new sites for large $n$.

The number of sites discovered by a particle with initial location $x\neq \0$ in the one-type process is dominated by the number of sites discovered by a particle from $x$ in a one-type process started with only the particles at $x$ activated. This gives the statement for any given particle in the one-type process. In the two-type process, the number of sites discovered by a given particle with initial location $x$ is dominated by the number of sites discovered by a particle from $x$ in a one-type process (constructed based on the same vector $\Pi$) started with only the particles at $x$ activated and where the $x$-particles move according to trajectories with the larger jump probability $p_1$ while all other particles use the smaller jump probability $p_2$. It follows from the same argument as above that this number is almost surely finite.
\end{proof}

\begin{remark}\label{rem1}
As noted in the proof, the same argument yields the same conclusion for a given particle also in a slightly modified one-type process where a finite number of particles move according to random walks with jump probability $p_1$ while the rest of the particles move according to random walks with jump probability $p_2<p_1$.
\end{remark}

\begin{proof}[Proof of Proposition \ref{pr:both_can_win}]
We prove that $\PP_{\0,z}(G_1^c\cap G_2)>0$, that is, the (possibly) weaker type 2 has a strictly positive probability of winning. That $\PP_{\0,z}(G_1\cap G_2^c)>0$ is proved similarly. We first treat the case when $p_1<1$ so that both types are lazy, and then describe how the argument can be generalized to the case when $p_1$ (and possibly also $p_2$) is equal to 1.

Consider a modified one-type process started with the particles at $\0$ and $z$ active at time 0, and where the particles starting at $\0$ move according to random walks with jump probability $p_1$, while all other particles (including those activated by the particles at $\0$) move according to random walks with jump probability $p_2$. The process is generated using the random objects in the vector $\Pi=(\nu,S,L)$, as described in Section \ref{sec:definition}. We let $\Pi^{\sss \rm{one}}_n$ denote the state of this process after $n$ steps, including the location and origin of all particles.

By Remark \ref{rem1}, the particles at $\0$ discover an almost surely finite number of sites in the above one-type process. With $N$ denoting the last time in the process a particle starting at the origin discovers a new site, we can hence pick $m$ such that $\PP(N\leq m)\geq 1/2$. Note that the set of discovered sites after $m$ steps is contained in $\cD_m=\{x\in\RR^d:\|x\|_1\leq m\}$ and write $v_m$ for the number of sites in $\cD_m$.

Now consider a two-type process started with the particles at $\0$ and $z$ active of type 1 and type 2, respectively. We will define coupled random walks $\hat{S}$ and delay variables $\hat{L}$ with the same distribution as $S$ and $L$ such that, if the two-type process is generated by $\hat{\Pi}=(\nu,\hat{S},\hat{L})$, then with positive probability the only particles that become activated by type 1 are those at $\0$. Essentially, the idea is to let the $\0$-particles stay put while type 2 progresses beyond the set of discovered sites in $\Pi^{\sss \rm{one}}_m$, preventing type 1 from discovering new sites if the $\0$-particles do not do so in the one-type process. To this end, the delay variables for the $\0$-particles before their first jump are generated independently for $k=1,\ldots,2v_m+m$, that is, for all $i=1,\ldots,\eta(x)$, we let
$$
\hat{L}_{0,k}^{\0,i}=\left\{
\begin{array}{ll}
\tilde{L}_{0,k}^{\0,i} & k=1,\ldots,2v_m+m;\\
L_{0,k-2v_m-m}^{\0,i} & k>2v_m+m,
  \end{array}
            \right.
$$
where $\{\tilde{L}_{0,k}^{\0,i}\}$ are i.i.d.\ uniform on $[0,1]$ and independent of $\{L_{0,k}^{\0,i}\}$. Furthermore, the variables controlling whether or not a given particle at $z$, say $(z,1)$, will jump in the time step after its $j$th jump are generated independently for $j=0,\ldots,2v_m-1$, that is,
$$
\hat{L}_{j,k}^{z,1}=\left\{
\begin{array}{ll}
\tilde{L}_{j,k}^{z,1} & j=0,\ldots,2v_m-1 \mbox{ and }k=1;\\
L_{j,k}^{z,1} & \mbox{otherwise},
  \end{array}
            \right.
$$
where $\{\tilde{L}_{j,k}^{z,1}\}$ are independent of $\{L_{j,k}^{z,1}\}$ with the same distribution. Also the jumps $j=1,\ldots,2v_m$ for the particle $(z,1)$ are generated by an independent random walk, that is,
$$
\hat{S}_j^{z,1}=\left\{
\begin{array}{ll}
\tilde{S}_j^{z,1} & j=1,\ldots,2v_m;\\
S_{j-2v_m}^{z,1} & j>2v_m,
  \end{array}
            \right.
$$
where $(\tilde{S}_j^{z,1})$ is independent of $(S_j^{z,1})$. All other particles move according to the same random walk trajectories as in $S$ and use the variables in $L$ to control their jumps. Note that $\hat{\Pi}$ has the same distribution as $\Pi$.

We now define two events that will guarantee that, if $N\leq m$, then type 2 wins in the two-type process based on $\hat{\Pi}$. First let $\hat{A}_0$ denote the event that the type 1 particles at the origin stay put in the first $2v_m+m$ time steps. Hence, on $\hat{A}_0$, the only particles that are type 1 at time $2v_m+m$ in the process are those at $\0$. As for type 2, let $\hat{A}_z$ denote the event that, in the time interval $[1,2v_m]$, the type 2 particle $(z,1)$ jumps between the sites in $\cD_m$, making one jump in each time step, in such a way that all sites in $\cD_m$ are visited at least once and at time $2v_m$ the particle $(z,1)$ returns to $z$. The particles in $\cD_m$ that are then activated by type 2 immediately start moving according to the same dynamics as in $\Pi$. Any other type 2 particles at $z$ and the particles activated by them develop in the same way as in $\Pi$. This means that all discovered sites in $\Pi^{\sss \rm{one}}_m$ are discovered at time $2v_m$ in the two-type process, and all sites except $\0$ are discovered by type 2. Finally, in the time interval $[2v_m+1,2v_m+m)$, the growth of type 2 continues based on the same random objects as in $\Pi$.

To summarize, on the event $\hat{A}_0\cap \hat{A}_z$, all particles that were activated in $\Pi^{\sss \rm{one}}_m$ are activated at time $2v_m+m$ in the two-type process based on $\hat{\Pi}$. Furthermore, all particles except those at $\0$ are activated by type 2 and have gotten at least as far along their random walk trajectories as in $\Pi^{\sss \rm{one}}_m$. Now assume that $N\leq m$, that is, the $\0$-particles do not discover any new sites after time $m$ in the modified one-type process. Then, when the $\0$-particles start moving according to the same random walks as in $\Pi$ at time $2v_m+m$ in the two-type process, they will not discover any new sites. Hence
$$
\PP_{\0,z}(G_1^c\cap G_2)\geq \PP_{\0,z}(\hat{A}_0\cap\hat{A}_z|N\leq m)\PP_{\0,z}(N\leq m).
$$
The events $\hat{A}_0$ and $\hat{A}_z$ are defined in terms of finitely many random objects that are independent of the objects in $\Pi$, implying that $\PP(\hat{A}_0\cap \hat{A}_z|N\leq m)=\PP(\hat{A}_0\cap \hat{A}_z)>0$. Furthermore $\PP(N\leq m)\geq 1/2$ by the choice of $m$. We conclude that $\PP_{\0,z}(G_1^c\cap G_2)>0$, as desired.

When $p_1=1$, so that type 1 (and possibly also type 2) is not lazy and thereby can not stay put, the argument is modified as follows. Pick a neighboring site of the origin, say $\1$, and assume without loss of generality that $z\neq \1$ and that $z$ is not a neighbor of $\1$. Extend the definition of $N$ to include also any particles at $\1$ so that no particle from $\0$ or $\1$ discovers a new site after time $N$ in the one-type process. Then let the type 1 particles from $\0$ jump back and forth between $\0$ and $\1$ while type 2 progresses as described above. Any particles at $\1$ that are activated by type 1 jump back and forth between $\1$ and $\0$. This is achieved by modifying the random walks associated with the particles at $\0$ and $\1$ in the beginning of the time course. We then arrive at a configuration where type 2 has progressed beyond its state at time $m$ in the one-type process and where the type 1 particles from $\0$ and $\1$ are thereby prevented from discovering any new sites if they do not do so in the one-type process.
\end{proof}

\section{Proof of Proposition \ref{pr:initial}}\label{sec:prop_initial}

We proceed with proving that the choice of the starting site $z$ for type 2 is irrelevant for the possibility of mutual infinite growth for $p_1,p_2\in(0,1)$.

\begin{proof}
To verify the claim, we will use a technique commonly referred to as ``sticky coupling'': Two copies of the process, started from different sites, evolve side by side until they enter the same state. From that time point on, the same random variables are used to generate the further evolution of both copies, preventing them from separating thereafter. In contrast to the standard argument of this kind, in our case the two copies will not evolve independently until they meet, but the second copy will be gradually aligned with the first one until they finally reach the same state. The first copy will be started from sites where we know that coexistence is possible, and the second copy from sites where we wish to show that coexistence is possible.

First assume that $\PP_{\0,z}(G_1\cap G_2)>0$. In the first copy, type 1 then starts from  $\0$ and type 2 from $z$. Fix a shortest path $\Gamma$ from $\mathbf{0}$ to $z$ and label its sites according to the following rule: a site $v\in\Gamma$ is assigned label $i$ ($i=1,2$) if $\eta(v)>0$ and its particles are activated by type $i$, and label $0$ if $\eta(v)=0$; see Figure \ref{fig:gamma}. Let $c(\Gamma)\in\{0,1,2\}^{\Gamma}$ be the (random) string of labels and let $M$ denote the time when all sites on $\Gamma$ have been discovered.

\begin{figure}
\centering
\includegraphics[scale=1.2]{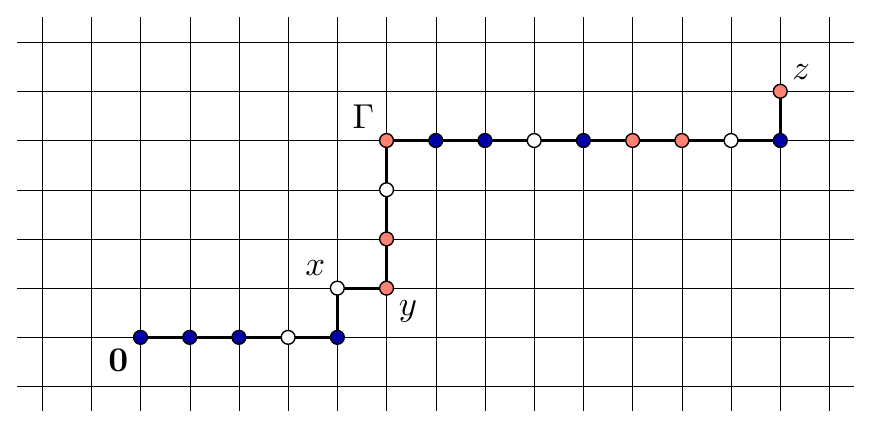}
\caption{Initially non-empty sites that are discovered by type 1 and type 2, respectively, are labeled 1 (blue) and 2 (red), and initially empty sites are labeled 0 (white).
\label{fig:gamma}}
\end{figure}	

Since $\{0,1,2\}^{\Gamma}$ is finite and $M$ is finite almost surely, our assumption implies that, for some $\gamma\in\{0,1,2\}^{\Gamma}$ and $m\in\N$ sufficiently large, the event  \[C_{\gamma,m}:=G_1\cap G_2 \cap \{M\leq m\}\cap\{c(\Gamma)=\gamma\}\] has positive probability. Let $y$ be the first site on $\mathbf{0}\stackrel{\Gamma}{\longrightarrow} z$ with label 2 (in $\gamma$) and $x$ its predecessor.

We now define a second copy of the competition process, with type 1 started in $x$ and type 2 started in $y$. Conditioned on $C_{\gamma,m}$, the second copy will reach the same state as the first copy in finite time with positive probability. To guarantee that there is a non-zero number of particles at $x$, we first change the initial configuration slightly by interchanging the number of particles at $\mathbf{0}$ and $x$. The process then starts with one type 1 particle traversing $\Gamma$ from $x$ to $\mathbf{0}$ and back. Next, one particle of each type, starting from $x$ and $y$, respectively, move along $\Gamma$ towards $z$ according to the following rule: The particle of type $i$ moves forward if either it is trailing or the label of the next site (attributed by $\gamma$) is in $\{0,i\}$, otherwise the type $i$ particle moves backwards. In this way, once the type 2 particle reaches $z$, all sites on $\Gamma$ have been activated by the type prescribed in $\gamma$. Finally, both particles return along $\Gamma$ to their initial position and the (type 1) particles placed at $\mathbf{0}$ and $x$ switch places. During all this time, no other activated particle than the ones specified moves. Note that the location of the particles now exactly corresponds to the starting configuration $\{\eta(v)\}_{v\in\Zd}$ of the first copy, however, all sites on $\Gamma$ have been activated.

At this point, each particle in the second copy is paired up one-to-one with a particle in the first copy, in such a way that the current position of the former and the initial position of the latter coincide. To couple the copies, we proceed as follows: All particles in the second copy mimic the moves of their twin in the first copy and, until the twin gets activated, the particles on $\Gamma\setminus\{\mathbf{0},z\}$ stay put by being lazy. Once all sites on $\Gamma$ are activated in the second copy, both copies are in the exact same state and further evolve identically. Since we manipulated only finitely many sites, particles and moves in the second copy, the coupling shows that $\PP_{\mathbf{0},z}(C_{\gamma,m})>0$ implies $\PP_{x,y}(G_1\cap G_2)>0$ and hence $\PP_{\mathbf{0},\mathbf{1}}(G_1\cap G_2)>0$ by rotation and translation invariance.

Now assume that $\PP_{\0,\1}(G_1\cap G_2)>0$. To show that $\PP_{\0,z}(G_1\cap G_2)>0$, we proceed in a similar fashion: In the first copy, type 1 is now started from $\mathbf{0}$ and type 2 from $\mathbf{1}$. For $n\in\N$, let $M_n$ denote the time when all sites in $\cD_n\cap \Zd$ have been discovered. Furthermore, for $v\in\Zd$, let $c(v)\in\{0,1,2\}$ denote the label attributed to $v$ according to the same rule as above, that is, $c(v)=0$ if $\eta(v)=0$ and $c(v)=i$ ($i=1,2$) if particles at $v$ are activated by type $i$. Since $\PP_{\0,\1}(G_1\cap G_2)>0$, for fixed $z$ the probability that $\cD_n$ contains at least $\lVert z\rVert_1$ sites with label 2 tends to $1$ as $n\to\infty$. Hence, as in the first part, we can choose first $n\geq\lVert z\rVert_1$, then $m$ big enough, such that for some $\lambda\in \{\text{0,1,2}\}^{\cD_n}$ and a collection of sites $\{y_1,\dots,y_{\lVert z\rVert_1}\}\subseteq\cD_n$, the event
\[C_{\lambda,m}:=G_1\cap G_2 \cap \big\{M_n\leq m\big\}\cap\big\{c(\cD_n)=\lambda\big\}\cap\Bigg(\bigcap_{k=1}^{\lVert z\rVert_1}\{c(y_k)=2\}\Bigg)\]
has positive probability.

Now consider a second copy started with type 1 in $\mathbf{0}$ and type 2 in $z$. In order to pair it up with the first copy (started with type 1 in $\0$ and type 2 in $\1$), we would first like a type 2 particle from $z$ to activate the site $\mathbf{1}$. This however, potentially causes incorrect labels on its way, which forces us to make some extra effort: Fix a shortest path $\Gamma:\ \mathbf{1}\to z$ and observe that $\Gamma\subseteq\cD_n$. Set $x_1=z$ and let $\{x_2,\dots,x_k\}$ be the sites on $\Gamma\setminus\{z\}$ that have label 1 in $\lambda$. Again, we alter the initial placement of particles in the second copy, this time by interchanging the numbers of particles initially placed at sites $x_i$ and $y_i$, for all $1\leq i\leq k$.
	
We now let the second copy evolve as follows: First one type 2 particle moves from $z$ to $\mathbf{0}$ via $\Gamma$ and the edge $\langle \mathbf{0},\mathbf{1}\rangle$. This type 2 particle then explores all sites of $\cD_n$ together with a type 1 particle from $\0$ in such a way that every site $v\in\cD_n\cap \Zd$ that has label $i$ in $\lambda$ is activated by the type $i$ particle, with the exception of the sites $y_1,\dots, y_k$, which are activated by the type 1 particle instead. All other activated particles (besides the pair activating $\cD_n$) idle by being lazy. Once all sites in $\cD_n$ have been discovered, the activating pair moves back to their initial positions. Furthermore, all (type 2) particles now placed at $x_i$ move to $y_i$ and all (type 1) particles placed at $y_i$ move to $x_i$, $1\leq i\leq k$, no particle ever leaving $\cD_n$. In this way, we have once more established a specified activation pattern (here $\lambda$) which occurs with positive probability in the first copy, given coexistence of both types, and moved all particles activated in the second copy to correspond to the initial configuration of the first copy.
	
Now we couple the copies as before: Activated particles in the second copy idle until their twin in the first copy gets activated and then mimic its moves. Conditioned on $C_{\lambda,m}$ for the first copy, after a finite time the two copies are in the exact same state, verifying that $\PP_{\0,z}(G_1\cap G_2)>0$.
\end{proof}

The above argument does not immediately extend to the case with $p_1=1$, since the particles can then not stay put. However, it turns out that the argument for one of the implications can be modified slightly so that it partially extends to the case with non-lazy particles. This will be important in obtaining Theorem \ref{th:coex} for $p_1=p_2=1$.

\begin{lemma}\label{le:odd_z}
Assume that either (i) $p_2<p_1=1$ or (ii) $p_1=p_2=1$ and $\|z\|_1$ is odd. Then $\PP_{\0,z}(G_1\cap G_2)>0$ implies that $\PP_{\0,\1}(G_1\cap G_2)>0$ for any initial distribution $\nu$.
\end{lemma}

\begin{proof}
We simply have to carefully go through the argument used in the proof of the first part of Proposition 1.2 and note that it generalizes to the case of non-lazy particles. The major problem arising is that particles of a non-lazy type cannot be forced to stay put. We will therefore assign to every site $v\in\Gamma$ a neighboring site $v'\in\Gamma$ and let non-lazy particles which in the evolution of the second copy are supposed to idle at $v$ instead jump back and forth between $v$ and $v'$: For a site $v$ on $\mathbf{0}\stackrel{\Gamma}{\longrightarrow} x$ we choose $v'$ to be its successor, and for $v$ on $y\stackrel{\Gamma}{\longrightarrow} z$ its predecessor. This way, we can still establish the prescribed activation pattern $\gamma$ on $\Gamma$ as above, since no site will be activated by particles jumping back and forth.

In the phase after $\Gamma$ has been activated and the copies are gradually coupled, a parity issue might arise in the above construction: Until its twin (initially placed at $v\in\Gamma$) in the first copy gets activated, the corresponding non-lazy particle in the second copy moves between $v$ and $v'$. For the coupling to work, all non-lazy particles jumping back and forth in the second copy have to be in their associated position $v$ once the twin gets activated at $v$ in the first copy. It is crucial to observe that, given our construction of the second copy, this is the case if and only if the $L_1$-distance between the two starting sites of a non-lazy type in the first and second copy, respectively, is even, owing to the fact that all particles of a non-lazy type at an odd (even) time will be at odd (even) $L_1$-distance to the site this type started from.

This settles the case in which only type 1 is non-lazy $(p_2<p_1=1)$ and $\lVert x\rVert_1$ is even. If $\lVert x\rVert_1$ is odd, we can fix the parity issue by starting the second copy instead with $\eta(\mathbf{0})$ active type 1 particles in $y$ and $\eta(y)$ active type 2 particles in $x$, which all move across the edge $\langle x,y\rangle$ in the first time step. Then we proceed as described above to conclude that $\PP_{y,x}(G_1\cap G_2)>0$, which again implies the claim by rotation and translation invariance.

In the case of two non-lazy types $(p_1=p_2=1)$, it is crucial for our construction that $\lVert z\rVert_1$ is odd, so that either $\lVert x\rVert_1$ being even or switching starting positions $x$ and $y$ in the second copy guarantees that for both types, the starting positions in the first and second copy share parity.
\end{proof}

\section{Proof of Theorem \ref{th:coex}}\label{sec:th_proof}

In this section we prove Theorem \ref{th:coex}. The same argument was used in \cite{GM_coex} by Garet and Marchand to prove coexistence in the two-type Richardson model, and it has later been used in  \cite{cont_coex} to prove an analogous result for a continuum model. It is also described in \cite[Section 4]{pleasures}. Here we combine it with Lemma \ref{le:finite}.

Before proceeding with the proof, we define the passage time $T(x,y)$ between two sites $x,y\in\Zd$ to be the time when the site $y$ is discovered in a one-type process started with the particles at $x$ active at time 0. Note that $T(x,y)=\infty$ if there are no particles at $x$, that is, if $\eta(x)=0$. It is not hard to see that these times are subadditive in the sense that
\begin{equation}\label{eq:subad}
T(x,y)\leq T(x,w)+T(w,y)\mbox{ for all }x,y,w\in\Zd.
\end{equation}
This is crucial in the proof of the shape theorem. Write $\n=(n,0,\ldots,0)$. Specifically, as shown in \cite{frogs_shape,frogs_shape_random}, it follows from subadditive ergodic theory that there exists a constant $\mu>0$ such that, conditional on $\eta(\0)\geq 1$,
\begin{eqnarray}\label{eq:mu}
\frac{T(\0,\n)}{n}\to \mu \mbox{ a.s. and in }L_1.
\end{eqnarray}

In order to handle initial distributions with empty sites, we will have to control the effect on $T(0,y)$ of conditioning on the presence of particles at some third site $x$. To this end we will need that, with large probability, the set of discovered sites in a one-type process contains some linearly growing ball, as stated in the below lemma. This is proved in a slightly more general formulation for non-lazy random walks in \cite[Lemma 2.5]{frogs_shape_random}. We give the general formulation and a brief explanation of why the result applies also to a lazy process in the appendix.

\begin{lemma}\label{le:liten_boll}
Consider the one-type frog model with initial distribution $\nu$ and $p\in(0,1]$. There exist constants $\tau\in(0,1)$ and $\alpha,\beta>0$ such that, conditional on $\eta(0)\geq 1$ and for all $n$:
\begin{equation}
\PP\big(\bar{\xi}_n\supseteq \cD_{\tau n}\big)\geq 1- \alpha\,\exp(-n^\beta).
\end{equation}
\end{lemma}

Now consider passage times based on $(\nu,S,L)$ and, for $x_1,\ldots ,x_k\in\Zd$, write $\E^{x_1,\ldots,x_k}$ for expectation conditional on $\eta(x_j)\geq 1$ for $j=1,\ldots,k$.

\begin{lemma}\label{le:non_zero} 
Consider the one-type frog model with $\PP(\eta(w)=0)>0$ but $\E[\eta(w)]<\infty$. For any $x,y\in\Zd$ we have that $\E^{\0,x}[T(\0,y)]\geq \E^{\0}[T(\0,y)]-C$, where $C$ is a positive constant that does not depend on neither $x$ nor $y$.
\end{lemma}

\begin{proof}
Let $\E^{\0,\neg x}$ denote expectation conditional on $\eta(\0)\geq 1$ and $\eta(x)=0$. We show that 
$$
\E^{\0,\neg x}[T(0,y)]\leq \E^{\0,x}[T(\0,y)]+C',
$$
where $C'$ does not depend on neither $x$ nor $y$. Since $\E^{\0}[T(\0,y)]$ is a convex combination of $\E^{\0,x}[T(\0,y)]$ and $\E^{\0,\neg x}[T(\0,y)]$ this gives the desired bound. We hence want to quantify the delay in a process without particles at $x$ compared to a process with particles at $x$. Note that this delay is bounded from above by the time when all sites that are discovered by particles originating from a non-empty $x$ have been discovered in a process without particles at $x$. This time, in turn, is stochastically dominated by the time when all sites that are discovered by the origin particles in a process started from the origin with a non-zero number of particles have been discovered in another copy of the process started with one single particle at the origin. Write $U$ for this time.

Consider the particles initially located at the origin and write $V$ for the last time when one of them discovers a new site. Recall from Lemma \ref{le:SRW} that $S_n$ denotes a random walk and let $\tau$ be as in Lemma \ref{le:liten_boll}. For $v$ big enough, we then have that
$$
\PP(V\geq v|\eta(0)=k)\leq k\PP(S_n\not\in\cD_{n^{1-\varepsilon}}\mbox{ for some }n\geq v)+\PP(\xi_n\not\supseteq \cD_{\tau n}\mbox{ for some }n\geq v).
$$
It follows from Lemma \ref{le:SRW} and Lemma \ref{le:liten_boll}, respectively, that the probabilities on the right hand side are summable in $v$. Hence $\E[V|\eta(0)=k]\leq kC_1+C_2$ for some constants $C_1,C_2<\infty$ and, since $\E[\eta(0)]<\infty$, we conclude that $\E[V]<\infty$.

Now consider the second copy of the process started with only one particle at the origin and the related time $U$ defined above. Write $\tilde{\xi}_n$ for the set of discovered sites at time $n$ in a process with initial distribution $\PP(\tilde{\eta}(w)=1)=\PP(\eta(w)\geq 1)=1-\PP(\tilde{\eta}(w)=0)$ and let $\tau$ be as in Lemma \ref{le:liten_boll} for such a distribution. Then
$$
\PP(U\geq u)\leq \PP(V\geq \tau u)+\PP(\tilde{\xi}_u\not\supseteq\cD_{\tau u}).
$$
That the probabilities on the right hand side are summable in $u$ follows from $\E[V]<\infty$ and Lemma \ref{le:liten_boll}, respectively. Hence $\E[U]<\infty$, as desired.
\end{proof}

We are now ready to prove Theorem \ref{th:coex}.

\begin{proof}[Proof of Theorem \ref{th:coex}]
Consider a two-type process started with the particles at the origin type 1 and the particles at $\n$ type 2, where $n$ will be specified below. By Proposition \ref{pr:initial} (if $p_1=p_2<1$) and Lemma \ref{le:odd_z} (if $p_1=p_2=1$), it suffices to show that coexistence has a positive probability in this process. Assume for contradiction that $\PP_{\0,\n}(G_1\cap G_2)=0$. Then one of the types must have at least probability 1/2 of being the winner and we may without loss of generality assume that $G_1^c\cap G_2$ has probability at least 1/2 (note that, if the tie-breaker is fair, both types have probability exactly 1/2 of winning). The idea is to show that the passage time from $\n$ to $-\m$ is substantially larger than the passage time from $\0$ to $-\m$ for some large $m$. On the other hand, if type 2 is the winner, we obtain an estimate that contradicts this, since the passage time from $\n$ to $-\m$ must then be shorter than the passage time from $\0$ to $-\m$ for large $m$. 

Consider passage times based on $(\nu,S,L)$ and fix $\vep>0$. We first treat the case when the initial distribution allows for empty sites. For $x_1,\ldots ,x_k\in\Zd$, write $\PP^{x_1,\ldots,x_k}$ and $\E^{x_1,\ldots,x_k}$ for  probability and expectation, respectively, conditional on $\eta(x_j)\geq 1$ for $j=1,\ldots,k$. By \eqref{eq:mu}, for $C$ as in Lemma \ref{le:non_zero}, we can pick $n>\frac{2C}{\mu\vep}$ large enough such that
\begin{equation}\label{eq:ncond}
\E^\0[T(\n,\0)]\leq(1+\vep)n\mu\quad \mbox{and}\quad \PP^\0(T(\n,\0)<(1-\vep)n\mu)<\vep.
\end{equation}
It is straightforward to check that, for any event $B$ with $\PP^\0(B)\geq \alpha$, the inequalities in \eqref{eq:ncond} implies that
\begin{equation}\label{eq:cond_exp}
\E^\0\left[T(\n,\0)|B^c\right]\leq \left(1+\frac{3\vep}{1-\alpha}\right)n\mu.
\end{equation}
We now claim that
\begin{equation}\label{eq:lower_bd}
\E^{\0,\n}[T(\n,-\m)-T(\0,-\m)]\geq (1-\vep)n\mu
\end{equation}
for arbitrarily large $m$. To see this, note that, for any integer $k$, trivially
$$
\E^\0[T(\0,k\n)]=\E^\0[T(0,\n)]+\E^\0[T(0,2\n)-T(\0,\n)]+\ldots+\E^\0[T(\0,k\n)-T(\0,(k-1)\n)].
$$
Since $\E^\0[T(\0,k\n)]/k\to n\mu$ as $k\to\infty$, it follows that $\E^\0[T(\0,(k+1)\n)-T(\0,k\n)]\geq (1-\vep/2)n\mu$ for arbitrarily large $k$. Taking $\m=k\n$ and using invariance, we obtain that
$$
\E^\0[T(\0,(k+1)\n)-T(\0,k\n)]=\E^\n[T(\n,-\m)]-\E^\0[T(\0,-\m)].
$$
The latter expectation is trivially bounded from below by $\E^{\0,\n}[T(\0,-\m)]$ since conditioning on the presence of additional particles can only decrease passage times. For the former expectation, if the expected initial number of particles per site is finite, then we have by Lemma \ref{le:non_zero} and the choice of $n$ that $\E^\n[T(\n,-\m)]\leq \E^{\0,\n}[T(\n,-\m)] + n\mu\varepsilon/2$ and can conclude that \eqref{eq:lower_bd} holds for arbitrarily large $m$.

Now consider the symmetric two-type process. As described above, we are working under the assumption that $\PP_{\0,\n}^{\0,\n}(G_1^c\cap G_2)\geq 1/2$. By Lemma \ref{le:finite}, if type 1 activates only finitely many particles, then the number of sites discovered by type 1 is also almost surely finite. Hence
$$
\lim_{m\to\infty}\PP^{\0,\n}(T(\n,-\m)\leq T(\0,-\m))\geq \lim_{m\to\infty}\PP^{\0,\n}_{\0,\n}(-\m \mbox{ is discovered only by type 2})\geq 1/2.
$$
Now let $B=\{T(\n,-\m)\leq T(\0,-\m)\}$ and pick $m$ large such that \eqref{eq:lower_bd} holds and such that $\PP^{\0,\n}(B)\geq 1/4$. We then obtain that

\begin{eqnarray*}
\E^{\0,\n}[T(\n,-\m)-T(\0,-\m)] & \leq & \E^{\0,\n}[T(\n,-\m)-T(\0,-\m)|B^c]\PP^{\0,\n}(B^c)\\
& \leq & \frac{3}{4}\E^{\0,\n}[T(\n,-\m)-T(\0,-\m)|B^c].
\end{eqnarray*}

By subadditivity, we have that $T(\n,-\m)-T(\0,-\m)\leq T(\n,\0)$, and \eqref{eq:cond_exp} hence yields that
$$
\E^{\0,\n}[T(\n,-\m)-T(\0,-\m)]\leq \frac{3}{4}(1+4\vep)n\mu.
$$
If $\vep$ is small, this contradicts \eqref{eq:lower_bd}, and we conclude that $\PP_{\0,\n}(G_1\cap G_2)>0$, as desired.

For initial distributions without empty sites we note that the conditioning on some sites being non-empty is throughout superfluous and the proof then goes through without the comparison of $\E^\n[T(\n,-\m)]$ and $\E^{\0,\n}[T(\n,-\m)]$ provided by Lemma \ref{le:non_zero}, that required $\E[\eta(w)]<\infty$.
\end{proof}

\section*{Appendix}

Here we sketch how Theorem \ref{th:shape} and Theorem \ref{th:diamond} follow from the same arguments as in the proofs of their analogues for non-lazy processes in \cite{frogs_shape,frogs_shape_random}.

As for Theorem \ref{th:shape}, this is an extension of \cite[Theorem 1.1]{frogs_shape_random} to a lazy process. This, in turn is a generalization to arbitrary initial distributions of \cite[Theorem 1.1]{frogs_shape}, which is restricted to processes started with one particle per site. The proof in \cite{frogs_shape_random} has the same structure as that in \cite{frogs_shape}, but requires some non-trivial additions to deal with initial configurations with empty sites. Specifically, the notion of $m$-good initial configurations is introduced. For such configurations, which are shown to occur with high probability, the same arguments as in \cite{frogs_shape} can be applied. Here we content ourselves with noting that these modifications go through also for a lazy process, and move on to describe how the key arguments from \cite{frogs_shape} are modified for a lazy process.

Recall that $T(x,y)$ denotes the time when the site $y$ is discovered in a process started with the particles at $x$ active at time 0. We now include the laziness parameter $p$ in the notation and write $T_p(x,y)$ for the passage time when the particles jump with probability $p$ in each time step. As usual, the process is constructed using the randomness in $(\nu,S,L)$.

The key ingredient in the proof of the shape theorem in \cite{frogs_shape} is the subadditive ergodic theorem \cite{liggett}. This is applied to the passage times $\{T_1(x,y)\}_{x,y\in\Zd}$ to conclude that $T_1(\0,nx)/n$ converges almost surely and in $L_1$ to some constant $\mu(x)>0$ for each $x\in\Zd$. All conditions of the subadditive ergodic theorem are easy to verify, except the requirement that $\E[T_1(\0,\1)]<\infty$. The main challenge in \cite{frogs_shape} is to verify this and the key result is the following tail bound for $T_1(0,x)$, formulated in \cite[Theorem 3.2]{frogs_shape} and extended in \cite[Lemma 2.1]{frogs_shape_random} to include also $d=1$.\footnote{The analogue estimate for $m$-good random initial configurations is given in \cite[Lemma 2.2]{frogs_shape_random}.}

\begin{theorem}[Lemma 2.1 \cite{frogs_shape_random}]\label{th:shap_est}
Suppose that $\eta(x)=1$. For all $d\geq 1$ and all $x\in \Zd$, there exist constants $\alpha=\alpha(x,d)>0$ and $\beta=\beta(d)>0$ such that
$$
\PP(T_1(0,x)\geq m ) \leq \alpha\exp\{-m^{-\beta}\}
$$
for all $m$.
\end{theorem}

Given this estimate and the conclusion of the subadditive ergodic theorem, the shape theorem follows from standard arguments; see \cite{frogs_shape} for details.

We now describe how the above bound can be extended to the passage time $T_p(\0,\1)$ for a lazy process by comparing $T_p(\0,x)$ to its analogue $T_1(\0,x)$ in the process without laziness. To this end, couple the processes by constructing them from the same random walks $S$. The site $x$ is discovered at time $T_1(\0,x)$ in the process without laziness and we can hence fix a sequence of $T_1(\0,x)$ particle jumps leading up to the discovery of $x$. Consider the particle involved in the $i$th such jump, let $x_{i-1}$ be its location before the jump and write $D_i$ for the number of time steps that the particle remains at $x_{i-1}$ in the lazy process before performing the jump. By construction, the random variables $\{D_i\}_{i=1}^{T_1(\0,x)}$ are independent and geometrically distributed with parameter $p$. Furthermore, it is easy to check that
\begin{equation}\label{eq:co}
T_p(\0,x) \leq T_1(\0,x) + \sum_{i=1}^{T_1(\0,x)}D_i.
\end{equation}
For any $m\in \mathbb{N}$ we hence have that
\begin{eqnarray}
\label{eq:bd}\PP \left(T_p(0,x) \geq \frac{4m}{p} \right) & \leq & \PP \left(T_1(\0, x)\geq \frac{2m}{p}\right) +  \PP\left(\sum_{i=1}^{T_1(\0,x)}D_i \geq \frac{2m}{p} \right)\\
\nonumber& \leq & 2\PP \left(T_1(\0, x)\geq m\right) +  \PP\left(\sum_{i=1}^mD_i \geq \frac{2m}{p} \right),
\end{eqnarray}
where the last inequality follows by conditioning on whether $T_1(\0,x)\geq m$ or not in the last term in \eqref{eq:bd}, and dominating the first term with $\PP(T_1(\0, x)\geq m)$.

Now note that $\sum_{i=1}^mD_i$ has a negative binomial distribution with parameters $m$ and $p$. Hence $\PP\left(\sum_{i=1}^mD_i\geq \frac{2m}{p}\right)\leq \PP\left(Y\leq m\right)$, where $Y$ is binomially distributed with parameters $\lceil2m/p\rceil$ and $p$, and hence $\E[Y]\geq 2m$. A standard Chernoff bound for the binomial distribution yields that $\PP(Y\leq m)\leq e^{-cm}$ for some constant $c>0$ and all $m\in\N$. Using Theorem \ref{th:shap_est}, we conclude that
$$
\PP \left(T_p(0,x) \geq \frac{4m}{p} \right)\leq \alpha'\exp\{-cm^\beta\},
$$
where $\alpha'=\alpha'(x,d)$ and $\beta=\beta(d)$ are positive finite constants. This bound serves the same purpose as the one in Theorem \ref{th:shap_est} for a lazy process.

We now give a short sketch of the proof of Theorem \ref{th:diamond}, that is, we describe why the ``full diamond'' result, Theorem 1.2 in \cite{frogs_shape_random}, still holds with laziness. To this end, let us first revisit Lemma 2.5 in \cite{frogs_shape_random}, which states that in dimension $d\geq2$, there exist constants $\tau\in(0,1)$ and $\alpha,\beta>0$ only depending on $d$, such that, conditional on $\{\eta(0)\geq 1\}$, for all $n\in\N$ and $x\in\Z^d$:
\begin{equation}\label{eq:lazy2.5}
\PP\big(\cD_{\tau n}^x\subseteq \bar{\xi}_{n+T(0,x)}\big)\geq 1- \alpha\,\exp(-n^\beta),
\end{equation}
where $\cD_r^x=\{y\in\RR^d:\;\lVert x-y\rVert_1\leq r\}$. An inspection of the proof reveals that the bound applies also in $d=1$ (although this is not needed in \cite{frogs_shape_random}). This is formulated in our Lemma \ref{le:liten_boll} for $x=\0$.

The crucial observation, which allows us to transfer the results from the original process to the lazy one, is that all estimates on discovered sites and activated particles are merely down-sized by a constant factor $p$. Specifically, the key to the proof of Lemma 2.5 is the fact that, for a simple symmetric random walk $(S_n)_{n\in\N}$ on $\Z^d$, there exists $\beta>0$ such that for all $d\geq 1$ and $k\geq1$:
$$
\PP\big(|\{S_j, 0\leq j \leq k^{1/2}\}|\geq k^{1/4}\big)\geq\beta,
$$
This follows from standard estimates using that the expected range of $(S_n)_{n=1}^k$ is $\Theta(\sqrt{k})$ for $d=1$, $\Theta\big(\frac{k}{\log k}\big)$ for $d=2$ and $\Theta(k)$ for $d\geq3$. This, however, also holds for a lazy walk and together with the lazy version of Theorem \ref{th:shap_est}, we arrive at \eqref{eq:lazy2.5} for our setting.

In the proof of Theorem 1.2 in \cite{frogs_shape_random}, the estimate \eqref{eq:lazy2.5} is used to show that, for $\theta\in\big(\frac{\delta}{d},1\big)$, the probability that there exists a site $x\in \cD_{\tau n^\theta}$ with $\eta(x)\geq(4d)^n$ that has been activated by time $n^\theta$, is at least $1-\exp(-Cn^{\theta d-\delta})$, where $C=\big(\log (4d)\big)^{-\delta}$. Conditioned on the existence of such an $x$, it is not hard to conclude that $\cD_{n-2Cn^\theta}\subseteq\bar{\xi}_{n+n^\theta}$ with high probability. In order to mimic this argument, we need $\eta(x)\geq\big(\frac{4d}{p}\big)^n$ instead. Given the strong condition on the tail of $\eta$ in the assumptions of Theorem \ref{th:diamond}, however, we find that
$$
\PP\Big(\eta(x)<\big(\tfrac{4d}{p}\big)^n\Big)\leq1-
\Big(\log\big(\tfrac{4d}{p}\big)\Big)^{-\delta}\,n^{-\delta}
$$
and slightly tweaking the constants will do the job.

\end{document}